\documentclass[12pt]{amsart}

\usepackage{amsmath,amssymb,amsfonts}
\usepackage{graphicx,color}

\newtheorem{theorem}{Theorem}[section]

\theoremstyle{definition}
\newtheorem{definition}[theorem]{Definition}

\newtheorem{prop}[theorem]{Proposition}
\newtheorem{cor}[theorem]{Corollary}

\theoremstyle{remark}
\newtheorem{remark}[theorem]{Remark}

\newcommand{\Zeta}{\mathbb{Z}}
\newcommand{\R}{\mathbb{R}}
\newcommand{\C}{\mathbb{C}}
\newcommand{\Q}{\mathcal{Q}}

\newcommand{\w}{\omega}
\newcommand{\f}{\widetilde{f}}
\newcommand{\J}{\mathbb{J}_\Gamma}
\newcommand{\V}{\mathcal{V}}
\newcommand{\W}{\mathcal{W}}

\begin{document}
\sloppy
\title[Crystal Multiwavelets in $L^2(\mathbb{R}^d)$  \, ]{
 Crystallographic Multiwavelets in $L^2(\mathbb{R}^d)$ }

\author[U.~Molter]{Ursula Molter}
\address{Ursula Molter,
Departamento de Matem\'atica, Facultad de Ciencias Exactas y Naturales, Universidad de Buenos Aires, Ciudad Universitaria, Pabellon I, 1428 Capital Federal, Argentina and IMAS, CONICET, Argentina.}
\email{umolter@dm.uba.ar}

\author[A.~Quintero]{Alejandro Quintero}
\address{Alejandro Quintero,
Departamento de Matem\'atica, Facultad de Ciencias Exactas y
  Naturales,  Universidad Nacional de Mar del Plata, Funes 3350, 7600
  Mar del Plata, Argentina}
\email{aquinter@mdp.edu.ar}
\thanks{Partially supported by UBACyT 20020130100403BA	 and ANPCyT PICT2014-1480}


\begin{abstract}
We characterize the scaling function of a crystal Multiresolution Analysis in terms of the vector-scaling function for a Multiresolution Analysis associated to a lattice. We give necessary and sufficient conditions in terms of the symbol matrix in order that an associated crystal wavelet basis exists.
\end{abstract}

 \maketitle

\section{Introduction}

From the early introduction of wavelets it is apparent that the tiling
properties of the group of translations play a central role. The main idea
in wavelet-systems, is to move (translate) a \emph{small wave} throughout
the space in such a way that every point is reached. In order to obtain
reproducing systems, dilations of the wave are introduced. Precisely, a
\emph{wavelet system} of multiplicity $L$ in $L^2(\mathbb{R}^d)$ is a set of
functions $\{\psi^1, \dots, \psi^L\}$, and a dilation matrix $a \in \mathbb{R%
}^{d\times d}$, such that
\begin{equation*}
\left\{D^j_a\tau_k \psi^i: k\in\mathbb{Z}^d, j \in \mathbb{Z}, i=1, \dots,
L\right\}
\end{equation*}
is an orthonormal basis for $L^2(\mathbb{R}^d)$, where $D_a$ is the unitary
dilation operator in $L^2(\mathbb{R}^d)$, and $\tau_k$ is the translation by
$k \in \mathbb{Z}^d$.

Mallat \cite{M} and Meyer \cite{Me} provided the first systematic way to
construct orthonormal wavelet bases of $L^2(\mathbb{R})$ using the concept
of Multiresolution Analysis (MRA). They showed that for every MRA there
exists an associated orthonormal wavelet basis. The rich structure of MRA is
generated by another function (the scaling function) that satisfies a
certain self-similarity condition. The problem of constructing orthonormal
wavelet was then shifted to the problem of constructing MRAs. Using this
structure, Daubechies \cite{D} was able to prove the existence of compactly
supported orthonormal wavelets with arbitrary regularity on the line. After
these results, the theory was generalised into different directions.

In this sense, \emph{crystal wavelets} and its associated \emph{%
crystal multiresolution analyses} are a natural generalization of
classical wavelets and multiresolution analyses. In these systems, a
crystal group $\Gamma$ takes the place of the translations in
classical wavelets. To be more precise,  we have the following
definition \cite{alfcar,MT}:

\begin{definition}
Given a crystal group (see
Definition \ref{defi-crista} below) $\Gamma\subset \mathrm{Isom}(\mathbb{R}^d)$
and an expanding matrix $\mbox{$a:\R^d\rightarrow\R^d$}$; (i.e. all
eigenvalues of $a$, have modulus greater
than $1$), such that $a\Gamma a^{-1}\subset\Gamma$,   a \emph{crystal multiwavelet of multiplicity} $L$ (wavelet, for $L=1
$), is a set of functions $\Psi=\{\psi^i:i=1, \dots, L\}\subset L^2(\mathbb{R%
}^d)$ such that
\begin{equation*}
\Psi_{(\Gamma,a)}=\{D^j_aD_\gamma \psi^i:j\in\mathbb{Z}, \gamma\in\Gamma,
i=1, \dots, L\},
\end{equation*}
is an orthonormal basis for $L^2(\mathbb{R}^d)$, where for any invertible affine map $\gamma$, $D_\gamma:L^2(%
\mathbb{R}^d)\rightarrow L^2(\mathbb{R}^d)$ is the unitary operator $D_\gamma f(x) :=|\det {\rm lin}(\gamma)|^{-\frac{1}{2}}f(\gamma^{-1}x)$,
where ${\rm lin}(\gamma)$ is the linear part of $\gamma$.

 We will also say that $\Psi$
is a $(\Gamma, a)$-crystal multiwavelet (or wavelet) system.
\end{definition}

The action by elements of crystal groups offers a richer geometrical
structure than using only translations. This new geometric freedom allows us to
obtain basis elements supported on sets with different shapes and directions
than classical wavelets. These are desirable properties for bases used in
representation of multidimensional functions \cite{GLLWW}.


The theory of crystal wavelets is related to the theory of
Composite Dilation Wavelets (CDW), which have been introduced and studied in
a series of papers by K. Guo, D. Labate, W. Lim, G. Weiss and E. Wilson (see for  example  \cite{GLLWW}). A particular class of CDW, are the so called \emph{composite $AB$ wavelets}. They are defined as follows:
let $\mathcal{L}\subset \mathbb{R}^{d}$ be a full rank lattice, $A$, $B$ countable subsets
of $GL_{d}(\mathbb{R})$. Then $\Psi =\{\psi ^{i}:i=1,\dots ,m\}$ is a
composite $AB$ wavelet if the affine reproducing system $\mathcal{A}_{AB}(\psi )=\{D_{a}D_{b}T_{k}\phi \ : \ k\in \mathbb{Z}%
^{d},b\in B,a\in A\},$ is a basis or frame for $L^{2}(\mathbb{R}^{d}).$ Results on existence of these wavelets were obtained in \cite{BK,BS,KRWW} for the particular case in which $A=\{a^{j}:j\in\mathbb{Z}\}$, where $a$ is an expansive matrix, and $B$ is a finite group satisfying the crystal
condition $B\mathcal{L}=\mathcal{L}$.

Using Zassenhaus's Theorem \cite{Z} one
can show that $\mathcal{A}_{AB}(\psi)$ is in fact a special case of crystal wavelet systems. Therefore the study of crystal wavelet systems includes these composite $AB$ wavelets.

The aim of the present paper is to construct crystal multiresolution analyses (Definition \ref{def MRA}) and prove the existence of crystal wavelets associated to these constructions.

Our first result (Theorem~\ref{relacion}) relates the property of being $\Gamma -$%
refinable (Definition \ref{def ref}) for an arbitrary splitting
crystallographic group to being refinable with respect to a translation
group. This relation enables us to characterize all compactly supported function with
orthonormal $\Gamma -$translates that generate a crystal
multiresolution analysis of $L^{2}(\mathbb{R}^{d})
$, in terms of multiscaling vectors using only translates on the lattice $\Lambda$. In particular, any such a function is a solution of a crystal refinement equation.

We then show (Theorem~\ref{teorema}) when there exist corresponding \emph{crystal
multiwavelets} whose dilates and $\Gamma-$translates form an orthonormal
basis for $L^2(\mathbb{R}^d),$ providing a method for generating orthonormal
bases of $L^2(\mathbb{R}^d)$ associated to a crystal  group.


Let us point out that in \cite{BK}, Theorem 1, using methods from noncommutative abstract harmonic analysis, the problem of the existence of wavelet systems associated to general noncommutative groups has been completely settled. This result has then be specialized in \cite{MT} to the case of crystal groups. In both cases the proof relies on the cancelation property for a finite Von Neumann
Algebra which does not immediately yield a method to exhibit such a system. In fact, in \cite{MT} only some particular examples are constructed.

We instead use classical Fourier Analysis results and are able to provide tools to find the wavelet systems explicitly. For the special case of AB composite wavelets in \cite{BK} a similar approach was used.

\section{Crystal Groups}

Crystal groups (or space groups) are groups of isometries of $\R^d$, that generalize the notion of translations, to allow for different (rigid) movements in $\R^d$. Precisely:
\begin{definition}\label{defi-crista}
A \emph{crystal group} is a discrete subgroup $\Gamma \subset {\rm Isom}(\mathbb{%
R}^{d})$ such that ${\rm Isom}(\mathbb{R}^{d})/\Gamma$ is compact, where  ${\rm Isom}(\mathbb{R}^{d})$ is endowed with the topology of pointwise convergence.

Or equivalently, one can define a \emph{crystal group} to be a discrete subgroup $\Gamma \subset {\rm Isom}(\mathbb{%
R}^{d})$ such that there exists a compact {\em fundamental domain} $P$ for $\Gamma$.
\end{definition}

Intuitively, a crystal should have a bounded pattern that is repeated until it fills
up space, i.e. there exists a bounded closed set $P$ such that
\begin{equation*}
\bigcup_{\gamma\in\Gamma}\gamma(P)=\R^d \mbox{ and } \gamma(P^\circ)\cap\gamma'(P^\circ)\neq\emptyset \mbox{ then } \gamma=\gamma',
\end{equation*}
where $P^\circ$ is the interior of $P$.

This set is called {\em fundamental domain}, which corresponds to the fundamental domain for lattices, only that here its shape can be much more general.

Note that a particular case of crystal group is the group of translations on a lattice.

The theorem of Bieberbach \cite{B1} yields the following:
\begin{theorem}[Bieberbach] \label{B1}
Let $\gamma$ be a crystal subgroup of ${\rm Isom}(\mathbb{R}^{d})$. Then
\begin{enumerate}
\item
$\Lambda=\Gamma \cap {\rm Trans}(\R^d)$  is a finitely generated abelian group of rank $d$
which spans ${\rm Trans}(\R^d)$, and
\item \label{point-def} the linear parts of the symmetries $ad \Gamma \cong \Gamma/\Lambda$, the {\em point group} of $\Gamma$, is finite.
\end{enumerate}
\end{theorem}

Here ${\rm Trans}(\R^d)$ stands for translations of $\R^d$.

We will denote the point group of $\Gamma$ defined in \eqref{point-def} of Theorem~\ref{B1}  by $G$,
and call $(\Gamma, G, \Lambda)$ a crystal triple.
\begin{remark}
\
{\rm
\begin{itemize}
\item Note that the set $\Lambda$ is not empty by Bierberach's theorem \cite{B1}. Moreover, $\Lambda$ consists of translations on a lattice $\mathcal{L}$ which is isomorphic to $\Zeta^d$.

 We will denote by $L$ and $L^*$ the fundamental domains of the lattices $\mathcal{L}$ and its dual, $\mathcal{L}^*$ respectively. We will also make a slight abuse of notation and use $\Lambda$ and $\Lambda^*$ for $\mathcal{L}$ and $\mathcal{L}^*$.
 \item The Point Group $G$ of $\Gamma$ is a finite subgroup of $\textbf{O}(d)$, the orthogonal group of $\mathbb{R}^{d}$, that preserves the lattice of translations, i.e. $G\mathcal{L}=\mathcal{L}$.
\end{itemize}
}
\end{remark}

General results on crystal groups, can be found for example in
\cite{B1}, \cite{gru}, \cite{Z}.

One very important class of crystallographic groups, are the {\em splitting crystallographic groups}:
\begin{definition}\label{split}
$\Gamma$ is called a \emph{splitting crystal group} if it is the
semidirect product of the subgroups $\Lambda $ and $G$. In this case $\Gamma =G\ltimes\Lambda$ and for each
$\gamma, \widetilde{\gamma} \in \Gamma$, with $\gamma=(g_i,\tau_k)$ and $\widetilde{\gamma}=(g_j,\tau_l)$, we have $\widetilde{\gamma}\cdot\gamma=(g_jg_i,\tau_kg^{-1}_i\tau_lg_i)$ where $g_i, g_j\in G$, $\tau_k,\tau_l\in\Lambda$ and $\gamma(x)=g_i(x+k)$.
\end{definition}

Every crystal group is naturally embedded into a splitting group, and very often arguments for general groups can be relatively easy reduced to the splitting case and then be proved for that simpler case.

For simplicty of notation, for each $\gamma\in\Gamma$ we will use the notation $\gamma=(g_i,k)$ in stead of $(g_i,\tau_k)$. If $\gamma=(g_i,k)$ and $\widetilde{\gamma}=(g_j,l)$, then $\widetilde{\gamma}\cdot\gamma=(g_jg_i,k+g_i^{-1}(l))$.

From here on we will only consider splitting crystal groups.

\section{$\Gamma-$refinement function and Crystal Multiresolution Analyses.}

In order to obtain wavelet-type reproducing systems, the group condition on the translations is not essential. However, having an underlying group allows to use powerful mathematical tools. In particular, if one wants to ensure a regular (discrete and uniform) movement through space using a group of isometries, the only possibilities are the crystal groups (see \cite{Z}).

The classic wavelets use the simplest of such groups, when only translations are used. We will here turn our attention to more general crystal groups.

In order to complete the construction of crystal wavelets, we need to add a dilation operator to our crystal group and a special dilation matrix.

\begin{definition}
  Let $(\Gamma,G,\Lambda)$ be a splitting crystal triple. We will say that $a$ is a \emph{$\Gamma-$admissible matrix}, if $a$ is an expanding matrix (i.e. all eigenvalues have modulus larger than $1$) and $a\Gamma a^{-1}\subset\Gamma.$
\end{definition}

We will assume throughout this paper that the $\Gamma-$admissible matrix $a$ admits a choice of digits $D=\{d_0,...,d_{m-1}\}$ such that the unique compact set $\Q$ that satisfies $a(\Q)=\bigcup_{s=0}^{m-1} \Q+d_s,$ is a \emph{tiling} of $\R^d$, i.e. $\{\Q+\lambda\}_{\lambda\in\Lambda}$ cover $\R^d$,
where $\Lambda$ is the group of translations associated to $\Gamma$. Lagarias and Wang proved that such digits exist if $d=1,2,3$ or if $m=|\det(a)|>d$ \cite{LW2}.

We can now introduce the notion of $(\Gamma,a)$-multiresolution analysis (see also \cite{BK}, \cite{alfcar}, \cite{MT}), which is a natural
generalization of the classical Multiresolution Analysis definition.

\begin{definition}\label{def MRA}
Let $\Gamma$ be a crystal group and let $a$ be a matrix such that $a\Gamma a^{-1}\subset\Gamma.$ A sequence $\{\V_{j}\}_{j\in \mathbb{Z}}$ of closed subspaces of $L^{2}(%
\mathbb{R}^{d})$ is an \emph{orthogonal $(\Gamma ,a)$-multiresolution analysis} ($
(\Gamma ,a)-MRA$)\emph{ of multiplicity $n$}, if the following
conditions hold:

\begin{enumerate}
\item[(i)] $\V_{j}\subset \V_{j+1}$, for  $j\in \mathbb{Z}$.

\item[(ii)] $\V_{j+1}=D_{a}^{-1}\V_{j}$, for  $j\in \mathbb{Z}$.

\item[(iii)] $\bigcap_{j\in \mathbb{Z}}\V_{j}=\{0\}$ and $\overline{%
\bigcup_{j\in \mathbb{Z}}\V_{j}}=L^{2}(\mathbb{R}^{d})$.

\item[(iv)] There exists an $n-$tuple $\Phi =(\phi _{1}, \dots, \phi _{n})$ of
functions $\phi _{i}\in L^{2}(\mathbb{R}^{d})$, called the \emph{crystal-scaling
function vector}, such that $\{D_{\gamma }\phi _{i}:\ \gamma \in \Gamma ,\
i=1, \dots, n\}$ is an orthonormal basis of $\V_{0}$.
\end{enumerate}
\end{definition}

\begin{remark}
If $\{\V_j\}_{j\in\Zeta}$ is a $(\Gamma,a)$-MRA, then, since $\V_0\subset\V_1$, the crystal-scaling vector function $\Phi\in \V_1$. As $\V_1=D^{-1}_a\V_0$,
the set $\{D_a^{-1}D_{\gamma }\phi _{i}:\ \gamma \in \Gamma ,\
i=1, \dots, n\},$
is an orthonormal basis of $\V_1$. Then for each $i=1,...,n$
\begin{equation*}
\phi_i(x)=\sum_{\gamma\in\Gamma}\sum_{j=1}^n d_{i,j}^\gamma D^{-1}_aD_\gamma\phi_j(x)=\sum_{\gamma\in\Gamma}\sum_{j=1}^n |\det a|^{-\frac{1}{2}}d_{i,j}^\gamma D_\gamma\phi_j(ax),
\end{equation*}
We consider now, for each $\gamma\in\Gamma$, the matrices $d_\gamma=|\det a|^{-\frac{1}{2}}(d_{i,j}^\gamma)_{i,j=1,...,n}.$
Then we have
\begin{equation}\label{eq}
\Phi(x)=\sum_{\gamma\in\Gamma}d_\gamma\Phi(\gamma^{-1}(ax)),
\end{equation}
where $d_\gamma$ are $n\times n$ matrices.

This equation is a generalization of the classic \emph{refinement equation}. We will therefore use the following definition.
\end{remark}

\begin{definition}\label{def ref}
Let $(\Gamma,G,\Lambda)$ be a splitting crystal triple and $a$ a $\Gamma-$admissible matrix.
Consider the vector-valued function $f:\R^d\rightarrow\C^r,\
f(x)=(f_1(x), \dots, f_r(x))^T$. We will say that
$f$ is \emph{$\Gamma-$refinable}, or that $f$ satisfies a \emph{$\Gamma$-refinement
equation}, if there exist a finite number of $r\times r$ matrices
$d_{\gamma}$ such that
\begin{equation}\label{gref}
  f(x)=\sum_{\gamma \in \Gamma}d_{\gamma}f(\gamma^{-1}(ax)).
\end{equation}
\end{definition}

An interesting question is whether given a sequence $\{d_\gamma\}_{\gamma\in\Gamma}\subset\C$, there exists a unique function $f\in L^2(\R^d)$ such that, $f$ is a solution to the $\Gamma-$refinement equation (\ref{gref}) associated to the sequence $\{d_\gamma\}_{\gamma\in\Gamma}$.

\subsection{$(\Gamma,a)-$symmetry and the $\mathbb{J}_\Gamma$ subspace.}

In this section, we will introduce the notion of $(\Gamma,a)-$symmetry, and define a special closed subspace of $L^2(\R^d)$ that will be necessary to characterize the existence and uniqueness of solutions to equation (\ref{gref}).

\begin{definition}
Let $(\Gamma,G,\Lambda)$ be a splitting crystal triple and $G=\{g_1, g_2, \dots, g_r\}$, and suppose that $g_1=Id$. Let $a$ be a $\Gamma-$admissible matrix and let $\{c_k\}_{k\in\Lambda}$, with $c_k\in \C^{r\times r}$. We will say that the matrices $c_k$ have  \emph{$(\Gamma,a)-$symmetry}, if
\begin{equation*}
c_{i,j}^k=c_{1,\rho_i(j)}^{g^{-1}_{h_i}(k)}\quad \mbox{ for all $\quad i,j = 1,...,r\quad$ and $\quad k\in\Lambda$,}
\end{equation*}
where $h_i$ and $\rho_i$ are permutations of $\{1,..,r\}$ such that $g_{h_i}=ag_ia^{-1}$ for $i=1, ..., r$ and $g_{\rho_i(j)}=g^{-1}_{h_i}\circ g_j$ for each $i,j=1, ..., r.$
\end{definition}

\begin{remark}
  There exist matrices with this property. Just as an example, let $\{d_\gamma\}_{\gamma\in\Gamma}$ be a sequence of complex  numbers. If we look at $\gamma=(g_i,k)$, with $g_i \in G, k\in\Lambda$, for each $k \in \Lambda$ we define the matrix $\widetilde{c}_k$ by
  \begin{equation}\label{matck}
\widetilde{c}_k=(c_{i,j}^{k})_{i,j=1,...,r}=\left(d_{(g^{-1}_{h_i}\circ g_j,g^{-1}_j(k))}\right)_{i,j=1,...,r}.
\end{equation}
It easily follows that these matrices have $(\Gamma,a)-$symmetry.
\end{remark}

%
Let as before $(\Gamma,G,\Lambda)$ be a splitting crystal triple, with $G=\{g_1=Id,...,g_r\}$. We define the following closed subspace of $L^2(\R^d,\C^r)$ associated to $\Gamma$:
\begin{equation}\label{conj_J_tild}
    \widetilde {\mathbb{J}_\Gamma} :=\{F\in L^2(\R^d,\C^r):F=(f,f\circ g_2^{-1},...,f\circ g_r^{-1})\; \}.
\end{equation}
We have the following Proposition.
\begin{prop}\label{Sinv}
  Let $(\Gamma,G,\Lambda)$ be a splitting crystal triple, $G=\{g_1=Id,...,g_r\}$ and let $a$ be a $\Gamma-$admissible matrix. Let $\{c_k\}_{k\in\Lambda}$, with $c_k\in \C^{r\times r}$ be matrices having $(\Gamma,a)-$symmetry. Let $\widetilde{\J}$ be defined as above \eqref{conj_J_tild}, then the operator $S:\widetilde{\J}\rightarrow \widetilde{\J}$ given by $\displaystyle SF(x)=\sum_{k\in\Lambda}c_kF(ax-k)$ is well defined.
\end{prop}

\begin{proof}
  Let $F\in \widetilde{\J}$, $F=(f,f\circ g_2^{-1},...,f\circ g_r^{-1})$.  We sill show that $SF\in \widetilde{\J}$. We consider $SF=\widetilde{F}=(\f_1,...,\f_r)$. Then $\widetilde{F}(x)=\sum_{k\in\Lambda}c_kF(ax-k)$. Hence, for each $i=1,...,r$, $\f_i(x)=\sum_{k\in\Lambda}\sum_{j=1}^r c^k_{i,j}f\circ g^{-1}_j(ax-k)$. Note that since $g_1=Id$, $\f_1=\f_1\circ g_1^{-1}$. For  $2\leq i\leq r$, by definition of $\f_1$, we have
\begin{equation}\label{ftilde}
  \f_1\circ g_i^{-1}(x)=\sum_{k\in\Lambda}\sum_{j=1}^r c^k_{1,j}f\circ g^{-1}_j(ag_i^{-1}(x)-k).
\end{equation}
By hypothesis $a$ is a $\Gamma-$admissible matrix. Therefore $G=aGa^{-1}$, and for each $i$ there exists $g_{h_i}\in G$ such that $ag_ia^{-1}=g_{h_i}$ and so $g^{-1}_{h_i}a=ag^{-1}_i$. Furthermore, for each $i$, we have that $g_{h_i}G=G$, since $G$ is a group. Therefore there exist permutations $s_{i}$ and $\rho_i$ of the set $\{1, \cdots , r\}$ such that $g_{s_i(j)}=g_{h_i}\circ g_j$ and $g_{\rho_i(j)}=g^{-1}_{h_i}\circ g_j$, $j=1,\dots,r$. Note that, in fact $\rho_i\circ s_i(j) = s_i\circ \rho_i(j) = j$. Thus by \eqref{ftilde}  we have
\begin{align}\label{c*}
  \f_1\circ g_i^{-1}(x)&=\sum_{k\in\Lambda}\sum_{j=1}^{r}c_{1,j}^kf(g_j^{-1}(g_{h_i}^{-1}(ax)-k))=\sum_{k\in\Lambda}\sum_{j=1}^{r}c_{1,j}^kf(g_j^{-1}\circ g_{h_i}^{-1}(ax-g_{h_i}k))\nonumber\\
  &=\sum_{k\in\Lambda}\sum_{j=1}^{r}c_{1,j}^kf(g_{s_{i}(j)}^{-1}(ax-g_{h_i}k))=\sum_{l\in\Lambda}\sum_{j=1}^{r}c_{1,j}^{g^{-1}_{h_i}(l)}f(g_{s_{i}(j)}^{-1}(ax-l))\nonumber\\
  &=\sum_{l\in\Lambda}\sum_{u=1}^{r}c_{1,\rho_i(u)}^{g^{-1}_{h_i}(l)}f(g_u^{-1}(ax-l)),
\end{align}
where in the last equality we used that $\rho_i(u)=j$.
Now, by the $(\Gamma,a)-$symmetry of the matrices $c_k$, the definition of $\f_i$ and (\ref{c*}) we have
\begin{equation*}
  \f_1\circ g_i^{-1}(x)=\sum_{l\in\Lambda}\sum_{u=1}^{r}c_{1,\rho_i(u)}^{g^{-1}_{h_i}(l)}f(g_u^{-1}(ax-l))
  =\sum_{k\in\Lambda}\sum_{j=1}^r c_{i,j}^k (f\circ g^{-1}_j)(ax-k)=\f_i(x).
\end{equation*}
Therefore $\widetilde{F}\in \widetilde{\J}$.
\end{proof}

\subsubsection{The Joint Spectral Radius.}

The spectral radius of a square matrix $M$ is
\begin{equation*}
\rho(M)=\lim_{l\rightarrow\infty}\|M^l\|^{1/l}=\max\{|\lambda|\;:\; \lambda\mbox{ is an eigenvalue of }\;M\}.
\end{equation*}
For each $1\leq p\leq\infty$, the \emph{$p$-joint spectral radius} ($p$-JSR) of a finite collection of $s\times s$ matrices $\mathcal{M}=\{M_1,...,M_m\}$ is
\begin{equation}\label{radio}
\widehat{\rho}_p(\mathcal{M})=\displaystyle \left\{
                                  \begin{array}{ll}
                                    \displaystyle \lim_{l\rightarrow\infty}\left(\sum_{\Pi\in P_l}\|\Pi\|^p\right)^{1/pl}, & \hbox{$1\leq p<\infty$;} \\
                                   \displaystyle  \lim_{l\rightarrow\infty}\max_{\Pi\in P_l}\|\Pi\|^{1/l}, & \hbox{$p=\infty$,}
                                  \end{array}
                                \right.
\end{equation}
where
\begin{equation*}
P_0=\{Id\}\;\;\mbox{ and }\;\; P_l=\{M_{j_1}...M_{j_l}:1\leq j_i\leq m\}.
\end{equation*}
It is easy to see that the limit in (\ref{radio}) exists and is independent of the choice of norm $\|\cdot\|$ on $\C^{s\times s}$. Note that if $p\geq q$, then $\widehat{\rho}_p(\mathcal{M})\leq\widehat{\rho}_q(\mathcal{M})$.

Note that if there is a norm such that $\displaystyle \left(\sum_{j=1}^m\|M_j\|^p\right)^{1/p}\leq\delta$, then, by the definition of $\widehat{\rho}$, it is clear that $\widehat{\rho}_p(\mathcal{M})\leq\delta$.

\subsection{Existence and uniqueness of solutions of a $\Gamma-$refinement equation.}

Since we are seeking compactly supported solutions of the $\Gamma-$refinement equation, we will only consider refinement equations with finitely many non-zero coefficients. Let $\Lambda' \subseteq \Lambda$ be a finite subset of $\Lambda$.  Analogously as in Chapter 2 of \cite{CM}, it can be shown that  the support of a solution to the $\Gamma$-refinement equation will be related to the compact attractor $K_{\Lambda'}$ of the iterated function system $\{w_k(x)= a^{-1}x + k, k \in \Lambda'\}$. From hereon, we will always assume that the coefficients of the refinement equation are all zero, except for a finite number.

We have the following Theorem.


\begin{theorem}\label{existencia}
  Let $(\Gamma,G,\Lambda)$ be a splitting crystal triple, $G=\{g_1=Id,...,g_r\}$, $a$ a $\Gamma-$admissible matrix and $m=|\det a|$. We consider a finitely supported sequence $\{d_\gamma\}_{\gamma\in\Gamma}\subset\C,$ such that $\sum_{\gamma\in\Gamma}|d_\gamma|^{2}<m$, and let $\widetilde{c}_k$ be the matrices generated from the coefficients $\{d_\gamma\}_{\gamma\in\Gamma}$ using equation \emph{(\ref{matck})}. Let as before, $\Lambda'\subset \Lambda$ be the finite set such that $\widetilde{c_k} \not=0$, for $k \in \Lambda'$, and $K_{\Lambda'}$ be the attractor associated the iterated function system $\{w_k(x)= a^{-1}x + k, k \in \Lambda'\}$. Let $\J$ defined analogously as above, by
 \begin{equation} \label{conj_J}
 \J=\{F\in L^2(\R^d,\C^r):F=(f,f\circ g_2^{-1},...,f\circ g_r^{-1})\; \mbox{ and }\;\mathrm{supp}(F)\subset K_{\Lambda'}\}.
 \end{equation}
Then there exists a unique function $F\in\J$ that is a solution to the refinement equation $F(x)=\sum_{k\in\Lambda'}\widetilde{c}_kF(ax-k)\;\; \mbox{ a.e.}.$
\end{theorem}
The proof of this result is a consequence of Theorem 3.4 of \cite{CM}, since the condition on the sequence $d_\gamma$ guarantees the existence of a (unique) solution in $L^2(\R^d,\C^r)$ to the vector-refinement equation $F(x)=\sum_{k\in\Lambda'}\widetilde{c}_kF(ax-k)$. But by Theorem~\ref{Sinv}, $S: \widetilde{\J} \rightarrow \widetilde{\J}$. Further, note that if $F \in \widetilde{\J}$ has compact support, so has $SF \in \widetilde{\J}$. Moreover, if $F$ has support in $K_{\Lambda'}$, $SF$ also (see \cite{CM}). Hence, since $\J$ is a closed subspace of $L^2(\R^d,\C^r)$  the unique solution must lie in $\J$.

We are now ready to prove our main result, relating the $\Gamma-$refinability of a function $f$, to the refinability of a certain vector function $F\in\J$.
\begin{theorem}\label{relacion}
 Let $(\Gamma,G,\Lambda)$ be a splitting crystal triple, $G=\{g_1=Id,...,g_r\}$, $a$ a $\Gamma-$admissible matrix and $m=|\det a|$.  We consider a finitely supported sequence $\{d_\gamma\}_{\gamma\in\Gamma}\subset\C$ and the non-zero matrices matrices $\{\widetilde{c}_k\}_{k\in\Lambda'}\subset \C^{r\times r}$,  generated from the coefficients $\{d_\gamma\}_{\gamma\in\Gamma}$ using equation \emph{(\ref{matck})}.
\begin{enumerate}
  \item If $f:\R^d\rightarrow\C$ has compact support and is $\Gamma-$refinable with coefficients $d_\gamma$, then the function $F=(f,f\circ g_2^{-1},...,f\circ g_r^{-1})$ is $\Lambda'-$refinable with coefficients  $\{\widetilde{c}_k\}_{k\in\Lambda'}$, and the sequence $\{\widetilde{c}_k\}_{k\in\Lambda'}$ has $(\Gamma,a)-$symmetry.
  \item If $\sum_{\gamma\in\Gamma}|d_\gamma|^2<m$ and $F=(f_1,...,f_r)\in L^2(\R^d,\C^r)$ is the solution of the refinement equation associated to the matrices $\{\widetilde{c}_k\}_{k\in\Lambda}$, then $F\in\J$ and the function $f=f_1$ is the solution of the $\Gamma-$refinement equation associated  to the scalars $\{d_\gamma\}_{\gamma\in\Gamma}$, i.e., $f$ is solution of $f(x)=\sum_{\gamma\in\Gamma}d_\gamma f(\gamma^{-1}ax)$  a.e. $x\in\R^d.$
\end{enumerate}
\end{theorem}

\begin{proof}
\begin{enumerate}
\item For the first implication, we must find matrices $\widetilde{c}_k\in\C^{r\times r}$
such that $F(x)=\sum_{k\in \Lambda} \widetilde{c}_kF(ax+k),$ and the matrices $\widetilde{c}_k$ have $(\Gamma,a)-$symmetry.

Since $f$ is $\Gamma$-refinable there exist coefficients $d_\gamma$, such that, $f(x)=\sum_{k\in \Lambda}\sum_{j=1}^r d_{(g_j, k)}f(g_j^{-1}(ax)-k).$
Hence, for $g_i\in G$, and for each $i=1, \dots, r$ we have
\begin{equation}
f(g^{-1}_i(x))=\sum_{k\in \Lambda}\sum_{j=1}^r
d_{(g_j, k)}f(g_j^{-1}(ag^{-1}_i(x))-k).
\end{equation}
Analogously to the proof of the Proposition \ref{Sinv}, we obtain that
\begin{equation*}
(f\circ g^{-1}_i)(x)=\sum_{\widetilde{k}\in\Lambda}\sum_{u=1}^r d_{(g^{-1}_{h_i}\circ g_u, g^{-1}_u(\widetilde{k}))}f(g^{-1}_u(ax-\widetilde{k}).
\end{equation*}

Now we consider the matrices
\begin{equation}\label{cks}
\widetilde{c}_k=(c_{i,j}^{k})_{i,j=1,...,r}=\left(d_{(g^{-1}_{h_i}\circ g_j,g^{-1}_j(k))}\right)_{i,j=1,...,r},
\end{equation}
and we denote by $(\widetilde{c}_k)_h$ the $h$-th row of
$\widetilde{c}_k$ then
\begin{equation*}
f(g^{-1}_{h}(x))=\sum_{k\in\Lambda}(\widetilde{c}_k)_h F(ax-k),
\end{equation*}
and therefore $F(x)=\sum_{k\in\Lambda} \widetilde{c}_k F(ax-k).$

To finish the proof of (1) we need to prove that the matrices $\widetilde{c}_k$ have $(\Gamma,a)-$symmetry.
For each $i=1,...,r,$ we choose $h_i\in\{1,...,r\}$ such that $g_{h_i}=ag_ia^{-1}$. If  $g_1=Id$ then $g_{h_1}=Id$. In this way, the elements of the first row of each matrix $\widetilde{c}_k$ are given by $c_{1,s}^k=c_{(g_s,g^{-1}_s(k))}$, for each $s=1,...,r.$
Furthermore, for a fixed $i$, there exists a permutation $\rho_i$ of the set $\{1,...,r\}$, such that for each $j=1,...,r$, $g_{\rho_i(j)}=g_{h_i}^{-1}\circ g_j$. Then the element $c_{i,j}^k$ of the matrix $\widetilde{c}_k$, is given by
\begin{equation*}
c_{i,j}^k=d_{(g_{h_i}^{-1}\circ g_j,g_j^{-1}(k))}=d_{(g_{\rho_i(j)},g_{\rho_i(j)}^{-1}\circ g_{h_i}^{-1}(k))}=c_{1,\rho_i(j)}^{g^{-1}_{h_i}(k)},
\end{equation*}
therefore, the matrices $\widetilde{c}_k$ have $(\Gamma,a)-$symmetry.

\item As $\sum_{\gamma\in\Gamma}|d_\gamma|^2<m$ then by Theorem~\ref{existencia} there exists a unique function $F=(f_1,...,f_r)$ such that $F(x)=\sum_{k\in\Lambda}\widetilde{c}_kF(ax-k),$ and $F\in\J.$ We consider $f=f_1$ and will prove that $f$ verifies
$f(x)=\sum_{\gamma\in\Gamma}d_\gamma f(\gamma^{-1}ax)$ a.e. $x\in\R^d$.
As $F$ is $\Lambda'-$refinable, then
\begin{equation*}
f(x)=\sum_{k\in\Lambda}\sum_{j=1}^r \widetilde{c}^k_{1,i}f\circ g^{-1}_j(ax-k)=\sum_{l\in \Lambda}\sum_{j=1}^r
\widetilde{c}^{g_j(l)}_{1,j}f(g^{-1}_j(ax)-l)).
\end{equation*}
We define $d_{(g_i,l)} :=\widetilde{c}^{g_i(l)}_{1,i}$, and so
\begin{equation*}f(x)=\sum_{k\in\Lambda}\sum_{i=1}^r d_{(g_i,k)}f(g_i^{-1}(ax)-k)=\sum_{\gamma\in\Gamma} d_{\gamma}f(\gamma^{-1}(ax)),
\end{equation*}
and therefore $f$ is $\Gamma-$refinable.
\end{enumerate}
\end{proof}
As a Corollary we obtain the important following Theorem:
\begin{theorem}
Let $(\Gamma,G,\Lambda)$ be a splitting crystal triple, $G=\{g_1=Id,...,g_r\}$, and let $a$ be a $\Gamma-$admissible matrix with $m=|\det a|$.  We consider the finitely supported sequence $\{d_\gamma\}_{\gamma\in\Gamma}\subset\C$ such that $\sum_{\gamma\in\Gamma}|d_\gamma|^2<m$. Then, there exists a unique solution, $f\in L^2(\R^d)$, to the $\Gamma-$refinement equation associated to the sequence $\{d_\gamma\}_{\gamma\in\Gamma}$, i.e. $f$ satisfies $\displaystyle f(x)=\sum_{\gamma\in\Gamma}d_\gamma f(\gamma^{-1}ax).$
\end{theorem}

This result is important on our  purpose of obtain properties of a $(\Gamma,a)-$Multiresolution Analysis from a classical Multiresolution Analysis.
%
%
%
%
\subsection{$(\Gamma,a)-$MRA}
We are interested in studying and generate $(\Gamma,a)-$MRA associated to a function $\phi\in L^2(\R^d)$, and understand how such a  $(\Gamma,a)-$MRA relates to a classical MRA.

\begin{definition}
Let $\Gamma$ be a crystal group, and assume that $\phi\in
L^2(\R^d,\C^n)$ has orthonormal translates on $\Gamma$. Let
$\mathcal{V}_0$ be the closed linear span of the translates of the
component functions $\phi_i$,
\begin{equation}\label{defv0}
\V_0=\overline{\mathrm{span}}\{D_{\gamma}\phi_i(x)\}_{\gamma\in\Gamma,\
i=1, \dots, n}.
\end{equation}

Then, for each $j\in\Zeta$, define $\V_j$ to be the set of all dilations of functions in $\V_0$ by $a^j$, i.e.,
\begin{equation}\label{defvj}
\V_j=\{f(a^jx):f\in\V_0\}.
\end{equation}

If the subspaces $\{\V_j\}_{j\in\Zeta}$ defined in this way are a
crystal multiresolution analysis for $L^2(\R^d, \C^n)$ (see Definition~\ref{def MRA}) then we
say that it is the $(\Gamma,a)-$MRA \emph{generated} by $\phi$.
\end{definition}

The following result relates a $(\Gamma,a)$-MRA in the case $\phi:\R^d \rightarrow \C$, to a classical MRA and is an easy consequence from Theorem \ref{relacion}.

\begin{cor}\label{lema}
Let $(\Gamma,G,\Lambda)$ be a splitting crystal triple, $r=|G|$, $G=\{g_1, \dots, g_r\}$. Then
$\{\V_j\}_{j\in\Zeta}$ is a $(\Gamma,a)$-MRA of multiplicity 1, with
crystal-scaling function $\phi$ if an only if the function
$\Phi:\R^d\rightarrow\C^r$ defined by
\begin{equation}\label{phi}
\Phi(x)=(\phi\circ g^{-1}_1(x), \dots, \phi\circ g^{-1}_r(x))^T,
\end{equation}
is the scaling vector function for a MRA of multiplicity $r$, and the coefficients of the refinement equation of $\Phi$ have $(\Gamma,a)-$symmetry.
\end{cor}

Using the previous Corollary and Theorem 4.4 in \cite{CM} we have the
following Theorem that characterize those functions $\phi$ that generate an $(\Gamma,a)-$MRA of multiplicity $1$.

\begin{theorem}\label{genera}
Let $(\Gamma,G,\Lambda)$ be a splitting crystal triple, $r=|G|$ and
assume that $\phi\in L^2(\R^d,\C)$ is compactly supported and has
orthonormal translates on $\Gamma$, i.e. for each
$\gamma,\sigma\in\Gamma$, $\langle D_{\gamma}\phi,D_{\sigma}\phi\rangle=\int D_{\gamma}\phi(x)\overline{D_{\sigma}\phi(x)}dx=\delta_{\gamma,\sigma}.$
Let $\V_j\subset L^2(\R^d)$ for $j\in\Zeta$ be defined by
\emph{(\ref{defv0})} and \emph{(\ref{defvj})}. Then the following
statements hold.
\begin{enumerate}
\item[(1)] Properties (ii), (iv) of the Definition \ref{def MRA}, and $\bigcap_{j\in\Zeta}\V_j=\{0\}$ are satisfied.
\item[(2)] Property (i) of the Definition \ref{def MRA}, is satisfied if and only if $\phi$ satisfies a refinement equation of the form
$\phi(x)=\sum_{\gamma\in\Gamma'}d_{\gamma}D_{\gamma}\phi(ax),$
for some scalars $d_{\gamma}$ and some finite set $\Gamma'\subset\Gamma$ (see equation (\ref{eq})).
\item[(3)] If
\begin{equation}\label{Q}
\left|\widehat{\phi}(0)\right|^2=\frac{|L|}{r}.
\end{equation}
where $L$ is the fundamental domain for the lattice $\Lambda$,
then $\bigcup_{j\in\Zeta}\V_j$ is dense in $L^2(\R^d)$. Moreover,
if $\phi$ is refinable, then $\bigcup_{j\in\Zeta}\V_j$ is dense in
$L^2(\R^d)$ if and only if \emph{(\ref{Q})} holds.
\end{enumerate}
\end{theorem}

\begin{remark} Note that
$\frac{|L|}{r}$ in equation (\ref{Q}) is the measure of a fundamental
domain for the crystal group $\Gamma$.
\end{remark}

\subsection{Relation between crystal-MRA and classical MRA}

An important application of Theorem \ref{relacion} is that to each
$(\Gamma,a)-$MRA of multiplicity $1$, for $\Gamma$ a splitting crystal group,  one can associate a classical MRA
of multiplicity $r,$ where $r$ is the cardinal of $G$, the Point Group of $\Gamma.$

In \cite{CG} it has been show that to any classical MRA one can associate a wavelet system. The symbol matrix $M_0$ defined as the unique matrix-valued function such that
\begin{equation*}\widehat{\phi}(a^*\omega)=M_0(\omega)\widehat{\phi}(\omega), \mbox{ }\omega\in\R^d,
\end{equation*}
where $\phi$ is the scaling vector function for a MRA, is a central step in their construction.

In what follows we will build the crystal version of the symbol matrix $M_0$. This is the unique function satisfying
\begin{equation*}
\widehat{\Phi}(a^*\omega)=M_0(\omega)\widehat{\Phi}(\omega), \mbox{ } \omega\in\R^d,
\end{equation*}
where $\Phi$ is the function defined in (\ref{phi}).

Let $\Gamma$ be a splitting crystal group,  $r$ the cardinal of $G$ and $\{\V_j\}_{j\in\Zeta}$ an $(\Gamma,a)$-MRA of multiplicity 1,
with associated crystal-scaling function $\phi\in L^2(\R^d)$. Then, by (\ref{eq}) $\phi$
satisfies a $\Gamma-$refinement equation of the form $\phi(x)=\sum_{\gamma\in\Gamma}d_\gamma D_\gamma\phi(ax)$.

For $\gamma=(g_j,k)$, we have $D_\gamma\phi(x)=|\det g_j|^{-\frac{1}{2}}\phi(\gamma^{-1}(x))$. Since $g_j\in \textbf{O}(d)$ then $|\det g_j|=1$, therefore
$D_\gamma\phi(x)=\phi(\gamma^{-1}(x))$. Hence, if $\gamma=(g_j,0)$ then $D_\gamma(\phi(x))=\phi(g_j^{-1}x)$. To simplify notation we will use $D_{g_j}\phi$ for $D_{(g_j,0)}\phi$ for each $j=1,...,r.$

The Fourier transform of $\phi$ is given by
$\widehat{\phi}(\omega)=\sum_{\gamma\in\Gamma}d_\gamma \widehat{D_\gamma\phi(a\cdot)}(\omega)$. Let us calculate $\widehat{D_\gamma\phi(a\cdot)}(\omega)$.

Then
\begin{align}\label{transf2}
\widehat{D_\gamma\phi(a\cdot)}(\omega)&=\int_{\R^d}e^{-2\pi i x.\omega}\phi(g^{-1}_j(ax)-k)dx=\frac{1}{m}\int_{\R^d}e^{-2\pi i (a^{-1}(g_j(y+k))).\omega}\phi(y)dy\nonumber\\
&=\frac{1}{m}e^{-2\pi i k.(g_j^*((a^{-1})^*\omega))}\int_{\R^d}e^{-2\pi i y.(g_j^*((a^{-1})^*\omega))}\phi(y)dy \nonumber\\
&=\frac{1}{m}e^{-2\pi i k.(g_j^*((a^{-1})^*\omega))}\widehat{\phi}(g_j^*((a^{-1})^*\omega)).
\end{align}
Note that as $G\subset\textbf{O}(d)$ then, for each $g\in G$ we have that $g^*=g^{-1}$, where $g^*$ denotes the adjoint operator of $g$ and $m=|\det a|$. Then by (\ref{transf2})
\begin{equation*}
\widehat{D_\gamma\phi(a\cdot)}(\omega)=e^{-2\pi i (g_j(k)).((a^{-1})^*\omega)}D_{g_j}\widehat{\phi}((a^{-1})^*\omega).
\end{equation*}

Therefore
\begin{equation}\label{transf}
\widehat{\phi}(\omega)=\sum_{\gamma\in\Gamma}d_\gamma e^{-2\pi i (g_j(k)).((a^{-1})^*\omega)}D_{g_j}\widehat{\phi}((a^{-1})^*\omega),
\end{equation}
from wich, since $\Gamma$ is a splitting crystal group, we obtain that
\begin{equation*}\widehat{\phi}(\omega)=\frac{1}{m}\sum_{j=1}^r\sum_{h\in\Lambda}d_{(g_j,g^{-1}_j(h))}
e^{-2\pi ih.(a^{-1})^*\omega}
D_{g_j}\widehat{\phi}((a^{-1})^*(\omega)).
\end{equation*}

Let us fix $i$, since
$\widehat{D_{g_i}\phi}(\omega)=D_{g_i}\widehat{\phi}(\omega)$
we have that

\begin{align}\label{1}
D_{g_i}\widehat{\phi}(a^*\w)&=D_{g_i}\left(\frac{1}{m}\sum_{j=1}^r\sum_{h\in\Lambda}d_{(g_j,g^{-1}_j(h))}e^{-2\pi ih.\omega} D_{g_j}\widehat{\phi}(\omega)\right)\nonumber\\
&=
\frac{1}{m}\sum_{j=1}^r\sum_{h\in\Lambda}d_{(g_j,g^{-1}_j(h))}e^{-2\pi
ih.g_i^{-1}(\omega)}D_{g_i\circ g_j}\widehat{\phi}(\omega).
\end{align}
If we denote by $\sigma_i$ the permutation of the set $\{1, \dots, r\}$
such that $g_i\circ g_j=g_{\sigma_i(j)}$, then we have from \eqref{1}
\begin{align} \label{2}
D_{g_i}\widehat{\phi}(a^*\w) &= \sum_{j=1}^r \frac{1}{m}\left(\sum_{h\in\Lambda}d_{(g_j,g^{-1}_j(h))}e^{-2\pi ig_i(h).(\omega)}\right)D_{g_{\sigma_i(j)}}\widehat{\phi}(\omega) \nonumber \\
&= \sum_{j=1}^r \frac{1}{m}\left(\sum_{k\in\Lambda}d_{\left(g_j,g_{\sigma_i(j)}^{-1}(k)\right)}e^{-2\pi ik.(\omega)}\right)D_{g_{\sigma_i(j)}}\widehat{\phi}(\omega)\nonumber \\
&= \sum_{j=1}^r
\frac{1}{m}\left(\sum_{k\in\Lambda}d_{\left(g_{\sigma^{-1}_i(j)},g_j^{-1}(k)\right)}e^{-2\pi
ik.(\omega)}\right)D_{g_j}\widehat{\phi}(\omega).
\end{align}

We define the matrix-valued function
\begin{equation}\label{M}
M_0(\omega)=\frac{1}{m}\sum_{k\in\Lambda}\widetilde{c_0}_{k}e^{-2\pi
ik.\omega}\mbox{ with }\widetilde{c_0}_{k}=\left(d_{(g_{\sigma_i^{-1}(j)},g_j^{-1}(k))}\right)_{i,j=1, \dots, r},\mbox{ }
\widetilde{c_0}_{k}\in\mathbb{C}^{r\times r}.
\end{equation}

Futher, let $\phi$ be as before the vector function, $\Phi:\R^d\rightarrow\mathbb{C}^r$ such that
$\Phi=(\phi_1, \dots, \phi_r)^T$, with $\phi_j=D_{g_j}\phi$ for all
$j=1,..,r$. Note that $D_{g_i}\widehat{\phi}(a^*\omega)$ is equal
to the product of row $i$ of $M_0(\omega)$ with $\Phi(\omega)$, and
hence $\widehat{\Phi}(a^*\omega)=M_0(\omega)\widehat{\Phi}(\omega).$

\section{Existence of crystal wavelets.}

Assume that we have a $(\Gamma,a)-$MRA of multiplicity $1$.
In this section we will give conditions for the existence of an
orthonormal wavelet bases for $L^2(\R^d)$. Further, we provide a construction of these $(\Gamma,a)$-MRA wavelets.

Given a $(\Gamma,a)-$MRA we define, as usual, the subspaces $\W_j$ where $\W_j$ is the
orthogonal complement of $\V_j$ in $\V_{j+1}$, i. e.,
$\mathcal{W}_j\equiv\V_{j+1}\ominus\V_j,\ j\in\Zeta$. We seek a set of
functions in $\V_1,$ $\psi_1, \dots, \psi_l\in\V_1$ such that the system
\begin{equation*}
\{D_{\gamma}\psi_i\ : i=1, \dots, l,\ \gamma\in\Gamma\},
\end{equation*}
is complete and orthonormal in $\W_0$. If such a set of functions
exists, then the $(\Gamma,a)-$MRA structure will guarantee that
the set
\begin{equation}
\{D_a^kD_{\gamma}\psi_i : i=1, \dots, l,\ k\in\Zeta,\ \gamma\in\Gamma\},
\end{equation}
is an orthonormal basis of $L^2(\R^d)$.

Therefore, our task is reduced to finding $\psi_1, \dots, \psi_l\in \W_0$ such that the system $\{D_\gamma\psi_i : i=1,...l\}$ constitute an orthonormal basis for $\W_0.$


\subsection{Characterization of the subspace $\V_1$.}

Our purpose is to find a family of functions in $\W_0$ such that its translates on $\Gamma$ constitute an orthonormal basis for $\W_0$. As $\V_1=\W_0\oplus\V_0$ we are interested in obtaining a characterization of the subspace $\V_1$. For this, given a function in $\V_1$ we will study its Fourier transform.
From now on we assume that $(\Gamma,G,\Lambda)$ is a
splitting crystal triple and we are given a $(\Gamma,a)$-MRA of multiplicity $1$ with crystal-scaling function
$\phi\in L^2(\R^d,\C)$, where $a$ is a dilation matrix that satisfies $a\Gamma
a^{-1}\subset\Gamma$, with $|\det a|=m$. Recall that

If $\psi\in \V_1$, then $\psi(x)=\sum_{\gamma\in\Gamma} c_\gamma \phi(\gamma^{-1}(ax)).$
For each $\gamma=(g,k)\in\Gamma$ the Fourier transform of $D_{\gamma}\phi(a\cdot) $ is given by (\ref{transf2}). Since $ g_i(\Lambda) = \Lambda $, for each $ i = 1, ..., r $ we have that
\begin{equation*}
\widehat{\psi}(\omega)=\sum_{i=1}^r\frac{1}{m}\sum_{h\in\Lambda}c_{(g_i,g^{-1}_i(h))}e^{-2\pi ih.(a^{-1})^*\omega} D_{g_i}\widehat{\phi}((a^{-1})^*(\omega)).
\end{equation*}

For each $i=1, \dots, r$ we call
$(N_\psi)_i(\omega)=\frac{1}{m}\sum_{h\in\Lambda}c_{(g_i,g^{-1}_i(h))}e^{-2\pi
ih.\omega}$. So, for each function $\psi\in \V_1$ we have
\begin{equation*}
\widehat{\psi}(\omega)=\sum_{i=1}^r(N_\psi)_i((a^{-1})^*\omega)D_{g_i}\widehat{\phi}((a^{-1})^*\omega), \mbox{ or }\; \widehat{\psi}(a^*\omega)=\sum_{i=1}^r
(N_\psi)_i(\omega)D_{g_i}\widehat{\phi}(\omega).
\end{equation*}

If $N_{\psi}(\omega)=(N_{1,\psi}(\omega), \dots, N_{r,\psi}(\omega))$, where $N_{i,\psi}=(N_\psi)_i$ for $i=1, \dots, r$,
and $\Phi(\omega)$ is the function defined by (\ref{phi}), then
\begin{equation*}\widehat{\psi}(a^*\omega)=N_{\psi}(\omega)\widehat{\Phi}(\omega)\quad \mbox{\rm and} \quad
\widehat{\psi}(\omega)=N_{\psi}((a^{-1})^*\omega)\widehat{\Phi}((a^{-1})^*\omega). \end{equation*}
Note that $N_\psi\in L^2(L^*,\mathbb{C}^{1\times r})$ with
$N_{\psi}=(N_{1,\psi}, \dots, N_{r,\psi})$ and $N_{i,\psi}$ is a
$\Lambda^*$ periodic function for $i=1, \dots, r$.

It is clear that the converse also holds, that is, if $N$ is a
$\Lambda^*$ periodic function and $N\in
L^2(L^*,\mathbb{C}^{1\times r})$, then the function $f$ defined by $\widehat{f}(\omega)=N((a^{-1})^*\omega)\widehat{\Phi}((a^{-1})^*\omega),$
belongs to $\V_1$. Therefore, we have proved the following characterization.

\begin{prop}\label{prop}
Let $f\in L^2(\R^d,\C)$. Then $f\in \V_1$, if and only if there
exist a $\Lambda^*$ periodic function $N$ and $N\in
L^2(L^*,\mathbb{C}^{1\times r})$ such that $\widehat{f}(a^*\omega)=N(\omega)\widehat{\Phi}(\omega)$.
\end{prop}

\subsection{Conditions for the existence of crystal Wavelets}
\

Let $(\Gamma,G,\Lambda)$  be a splitting crystal triple with $r=|G|$, and let us consider a $(\Gamma,a)-$MRA of multiplicitity $1$. Let $m=|\det a|$.

We are looking for the right conditions for the existence of a wavelet basis.
Since our goal is to find a basis of $\W_0$, as stated before, it is necessary to find $m-1$ functions to
generate this subspace, this condition is analogous to classical wavelets.

If $\psi\in\V_1$ then $
\psi(x)=\sum_{\gamma\in\Gamma}c^{\psi}_\gamma\phi(\gamma^{-1}(ax)).$
Analogously as for the construction of the symbol matrix $M_0$ (Section 2, formula (\ref{M})), we can construct the matrix-valued function
\begin{equation}\label{matrix}
M_\psi(\omega)=\frac{1}{m}\sum_{k\in\Lambda}\widetilde{c}_{\psi,k}e^{-2\pi
ik.\omega},
\end{equation}
such that
\begin{equation}\label{rel}
\widehat{\Psi}(a^*\omega)=M_{\psi}(\omega)\widehat{\Phi}(\omega),
\end{equation}
where $\Psi=(\psi_1, \dots, \psi_r)$ with $\psi_i=D_{g_i}\psi$ for all $i=1, \dots, r,$ and
\begin{equation}\label{c}
\widetilde{c}_{\psi,k}=\left(c^{\psi}_{(g_{\sigma_i^{-1}(j)},g_j^{-1}(k))}\right)_{i,j=1, \dots, r}.
\end{equation}

Then to each function $\psi\in\V_1$ we associate a matrix $M_\psi$ of the form (\ref{matrix}) such that $\psi$ and $M_\psi$ satisfy an equality of type (\ref{rel}).

If we now have $m-1$ functions $\psi^1, \dots, \psi^{m-1}$ in $\V_1$ then
\begin{equation*}
\psi^l(x)=\sum_{\gamma\in\Gamma}c^l_{\gamma}\phi(\gamma(a^{-1}x))\mbox{ for each }l=1, \dots, m-1,
\end{equation*}
and hence for each $l=1, \dots, m-1,$ we have a matrix-valued function
\begin{equation*}
M_l(\omega)=\frac{1}{m}\sum_{k\in\Lambda}\widetilde{c}_{l,k}e^{-2\pi
ik.\omega},
\end{equation*}
where $\widetilde{c}_{l,k}$ are the matrices defined as
in (\ref{c}) with the coefficients $c^l_{(g_i,k)}$, for all
$l=1, \dots, m-1$, respectively.

Futhermore, if we consider the vector functions $\Psi_l:\R^d\rightarrow\mathbb{C}^r$ such
that $\Psi_l=(\psi^l_1, \dots, \psi^l_r)^T$, with $\psi^l_j=D_{g_j}\psi^l$ for
all $j=1,..,r$ then
\begin{equation*}
\widehat{\Psi}_l(a^*\omega)=M_l(\omega)\widehat{\Phi}(\omega), \mbox{ }l=1, \dots, m-1.
\end{equation*}

Note that $M_l$ can be decomposed into the different cosets of $\Lambda$, that is
\begin{equation*}
M_l=M_{l0}+\cdots+M_{l(m-1)}\mbox{ with } M_{lh}(\omega)=\frac{1}{m}\sum_{k\in\Lambda_h}\widetilde{c}_{l,k}e^{-2\pi ik.\omega}.
\end{equation*}
If we define
\begin{equation*}
  u_{lh}(\omega)=\frac{1}{\sqrt{m}}\sum_{k\in\Lambda}\widetilde{c}_{l,ak+d_h}e^{-2\pi ik.\omega}\mbox{ with } h=0,...,m-1,
\end{equation*}
then we obtain
\begin{equation*}
M_{lh}(\omega)=\frac{e^{-2\pi id_h.\omega}}{\sqrt{\omega}}u_{lh}(a^*\omega)\mbox{ and }M_l(\omega)=\sum_{h=0}^{m-1}\frac{e^{-2\pi id_h.\omega}}{\sqrt{\omega}}u_{lh}(a^*\omega).
\end{equation*}

It is now apparent that if we want to prove the existence of a multiwavelet vector function $\psi_l$, we need to find conditions on the matrices $M_1, \dots, M_{m-1}$ such that the translates on $\Gamma$ of $\{\psi^l\; : l=1, \dots, m-1\}$ will form an orthonormal basis for $\mathcal{W}_0$.

Recalling $M_0$ from (\ref{M}) we define the matrix valued function $\mathcal{M}(\omega)\in(\C^{r\times r})^{m\times m}$ by
\begin{equation*}
\mathcal{M}(\w)=\left[M_i(\w+(a^*)^{-1}\varrho_j)\right]_{i,j=0, \dots, m-1},
\end{equation*}
where $\{\varrho_j\}_{j=0,...,m-1}$ is a complete set of representatives of $\widetilde{\Lambda}/a^*\widetilde{\Lambda}$, where $\widetilde{\Lambda}$ is the dual lattice of $\Lambda$. Note that $\mathcal{M}$ \emph{is unitary a.e.} if and only if for each $i,j=0, \dots, m-1$, we have
\begin{equation*}
\sum_{n=0}^{m-1}M_i(\w+(a^*)^{-1}\varrho_n)M^*_j(\w+(a^*)^{-1}\varrho_n)=\delta_{i,j}Id_{r\times r}.
\end{equation*}

Similarly we define the matrix
\begin{equation*}
\mathcal{U}(\phi,\psi_1,...,\psi_m-1)(\omega)=[u_{jh}(\omega)]_{j,h=0,...,m-1}.
\end{equation*}

\begin{definition}
  Let $(\Gamma,G,\Lambda)$ be a splitting crystal triple and $G=\{g_1,...,g_r\}$. Let $\{V_j\}_{j\in\Zeta}$ be a $(\Gamma,a)-$MRA of multiplicity $1$ with scaling function $\phi$, and consider $\{\widetilde{V}_j\}_{\j\in\Zeta}$ a MRA of multiplicity $r$ with vector scaling function $\Phi=(\phi\circ g_1^{-1},...,\phi\circ g_r^{-1})$. We define the set $\mathbb{V}_1\subset\widetilde{V}_1$ such that $F\in\mathbb{V}_1$ if $F\in \widetilde{V}_1$ and the coefficients $c_k$ of $F(x)=\sum_{k\in\Lambda}c_k\Phi(ax-k)$ have $(\Gamma,a)-$symmetry.
\end{definition}

\begin{prop}
  Let $(\Gamma,G,\Lambda)$ be a splitting crystal triple and $G=\{g_1,...,g_r\}$. Let $\{V_j\}_{j\in\Zeta}$ be a $(\Gamma,a)-$MRA of multiplicity $1$ with scaling function $\phi$. Then $F=(f_1,...f_r)\in\mathbb{V}_1$ if and only if $f_i=f_1\circ g_i^{-1}.$
\end{prop}

With this notation we have the following proposition:
\begin{prop}
  Let $(\Gamma,G,\Lambda)$ be a splitting crystal triple, $|G|=r$, and let $f_0, f_1, \dots, f_{m-1}$ be functions in $\mathcal{V}_1$. If we write
$
 S_l=\{D_\gamma f_j(x):j=0,...,l-1;\gamma\in\Gamma\} l=1, \dots, m,$
and for each  $j=0,...,m-1$ we consider $F_j=(f_j,f_j\circ g_2^{-1},\cdot,f_j\circ g_r^{-1})$, then we have,
\begin{enumerate}
\item The system $S_1$ is orthonormal if and only if $F_0\in\mathbb{V}_1$ and the functions $u_{0h}$ associated to $F_0$ satisfy $\sum_{h=0}^{m-1}u_{0h}u^*_{0h}=I_r$ a.e.

\item For each $l\leq m,$ the system $S_l$ is orthonormal if and only if $F_j\in\mathbb{V}_1$ for $j=0,...,l$ and $\sum_{h=0}^{m-1}u_{ih}u^*_{jh}=\delta_{ij}I_r$

\item The system $S_m$ is an orthonormal basis for $\V_1$ if and only if $F_j\in\mathbb{V}_1$ for $j=0,...,m-1$, and the block matrix $\mathcal{U}(F_0,...,F_{m-1})$ is unitary a.e.
\end{enumerate}
\end{prop}

In the following Theorem we will characterize the existence of crystal multiwavelets in terms of the matrix $\mathcal{M}$.

\begin{theorem}\label{teorema}
Let $\Gamma$ be a splitting crystal group, and let $\{\V_j\}_{j\in\Zeta}$ be a $(\Gamma,a)-$MRA for $L^2(\R^d)$ of multiplicity 1. Then, the following statements are equivalent.
\begin{enumerate}
\item[(a)] $\{D_{\gamma}\psi^l\}_{\gamma\in\Gamma,\ l=1, \dots, m-1}$ forms an orthonormal basis for $\mathcal{W}_0.$
\item[(b)] $\mathcal{U}$ is unitary a.e., and the functions $\Psi_l$ associated a $\mathcal{U}$ belongs to $\mathbb{V}_1$.
\item[(c)]$\mathcal{M}$ is unitary a.e., and the functions $\Psi_l$ associated a $\mathcal{M}$ belongs to $\mathbb{V}_1$.
\item[(d)] $\frac{1}{m}\sum_{k\in\Lambda}\widetilde{c}_{i,k}\widetilde{c}^*_{j,k-Av}=\delta_{0,v}\delta_{i,j}I_{r\times r}$ for $v\in\Lambda$ and $i,j=0, \dots, m-1,$ and the matrices $\widetilde{c}_{i,k}$ have a $(\Gamma,a)-$symmetry.
\end{enumerate}
\end{theorem}

\begin{proof}
Since $\Gamma$ is a splitting crystal group, then the system
\begin{equation*}
\{D_{\gamma}\psi^l\}_{\gamma\in\Gamma,\ l=1, \dots, m-1}=\{D_{g_i}\psi^{l}(x-k)\}_{k\in\Lambda,\ i=1, \dots, r,\ l=1, \dots, m-1}.
\end{equation*}
If we use the  notation, $\psi_{i}^l=D_{g_i}\psi^l$ as above,
\begin{equation*}
\{D_{\gamma}\psi^l\}_{\gamma\in\Gamma,\ l=1, \dots, m-1}=\{\psi_{i}^l(x-k)\}_{k\in\Lambda,\ i=1, \dots, r,\ l=1, \dots, m-1}.
\end{equation*}
Further, as $\{\V_j\}_{j\in\Zeta}$ is a $(\Gamma,a)-$MRA,
then by Corollary \ref{lema}, $\{\V_j\}_{j\in\Zeta}$ is a MRA of multiplicity $r$ with vector
scaling function $\Phi$. Therefore, the proof of the Theorem is immediate from Theorem 4.11 in \cite{CM}.
\end{proof}

Thus, once an $(\Gamma,a)$-MRA has been found, we can construct a wavelet basis for $L^2(\R^d)$ if we can construct a particular unitary matrix function $\mathcal{M}(\omega)$. For each $\omega$, the matrix $\mathcal{M}(\omega)$ is of size $rm\times rm$, and the first $r$ rows of this matrix are known. If the remaining rows can be completed so that $\mathcal{M}(\omega)$ is unitary a.e., then we can find the functions that generate the wavelet bases. Equivalently, we can try to solve the non-linear system of equations in $(d)$. The question of whether this completion can always be accomplished is a very difficult open question. The single function multivariate case, with dilation $2I_n$ is solved by the fundamental Lemma of Gr\"{o}chenig \cite{G}, and if $(m-1)r \geq \frac{d}{2}$, Cabrelli et.~al \cite{CG} proved that the completion can always be accomplished. This question is related to the {\em Extension Principle}, see recent work in \cite{APS16, FHS16}.

\end{document}